\documentclass[11pt, a4paper]{article}
\usepackage{t1enc}
\usepackage[latin1]{inputenc}
\usepackage[english]{babel}
\usepackage{amsmath,amsthm}
\usepackage{mathrsfs}
\usepackage{amsfonts}
\usepackage{latexsym}
\usepackage[dvips]{graphicx}
\usepackage{float}
\usepackage[natural]{xcolor}
\usepackage{algorithm}
\usepackage{algorithmic}
\usepackage{enumerate,enumitem}
\usepackage{multirow}
\usepackage{xcolor}
\usepackage[colorlinks,linkcolor=blue]{hyperref}
\usepackage{lineno}
\usepackage{setspace}
\usepackage{fullpage}
\usepackage[title]{appendix}
\usepackage{todonotes}
\usepackage{amssymb}
\usepackage{array}
\usepackage{listings}
\usepackage{bbm}
\usepackage{tikz}
\usetikzlibrary{shapes,positioning}
\usepackage{caption}
\usepackage{subfigure}
\usepackage{url}
\usepackage{fullpage}
\usepackage{hyperref}
\usepackage[margin=2.3cm]{geometry}
\usepackage{enumitem,verbatim}

\newtheorem{theorem}{Theorem}[section]
\newtheorem{claim}[theorem]{Claim}

\newtheorem{lemma}[theorem]{Lemma}

\newtheorem{corollary}[theorem]{Corollary}

\newtheorem{example}{Example}[section]

\theoremstyle{definition}

\newcommand{\ex}{{\rm ex}}

\usepackage{todonotes}
\title{Subdivision-free graphs with the maximum spectral radius}

\author{
Wanting Sun \thanks{Data Science Institute, Shandong University, Jinan, 250100, China. Email: {\tt wtsun@sdu.edu.cn}. Supported by the Natural Science Foundation of Shandong Province (ZR2024QA023) and the China Postdoctoral Science Foundation (2023M742092).}
\and
Guanghui Wang \thanks{School of Mathematics, and State Key Laboratory of Cryptography and Digital Economy Security, Shandong University, Jinan, 250100, China. Email: {\tt ghwang@sdu.edu.cn}. Supported by the National Key Research and Development Program (2023YFA1009603) and the Natural Science Foundation of China (12231018).
}
\and 
Pingchuan Yang \thanks{School of Mathematics, Shandong University, Jinan, 250100, China. Email: {\tt yyppcchh@126.com}.}
}
\date{\today}

\begin{document}
\maketitle
\begin{abstract}
    Given a graph family $\mathbb{H}$, let ${\rm SPEX}(n,\mathbb{H}_{\rm sub})$ denote the set of $n$-vertex $\mathbb{H}$-subdivision-free graphs with the maximum spectral radius. In this paper, we investigate the problem of graph subdivision from a spectral extremal perspective, with a focus on the structural characterization of graphs in ${\rm SPEX}(n,\mathbb{H}_{\rm sub})$. For any graph $H \in \mathbb{H}$, let $\alpha(H)$ denote its independence number. Define $\gamma_\mathbb{H}:=\min_{H\in \mathbb{H}}\{|H| - \alpha(H) - 1\}$. 
    We prove that every  graph in ${\rm SPEX}(n,\mathbb{H}_{\rm sub})$ contains a spanning subgraph isomorphic to $K_{\gamma_\mathbb{H}}\vee (n-\gamma_\mathbb{H})K_1$, which is obtained by joining a $\gamma_\mathbb{H}$-clique with an independent set of $n-\gamma_\mathbb{H}$ vertices.
    This extends a recent result by Zhai, Fang, and Lin concerning spectral extremal problems for $\mathbb{H}$-minor-free graphs. 

    \vspace{2mm}
    \textbf{Keywords:} graph subdivision; spectral radius; generalized book;  stability method
\end{abstract}
\maketitle

\section{Introduction}

Many important problems in extremal graph theory can  be framed as subgraph containment problems, which ask for conditions on a graph $G$ that ensure it contains a copy of a general graph $H$. Given a graph~$H$ and~$n\in\mathbb{N}$, the {\textit{Tur\'an number}} of~$H$, $\ex(n,H)$, is the maximum number of edges an~$n$-vertex~graph can have without containing a copy of~$H$. 
Mantel's theorem states that $\ex(n,K_3)=\lfloor n^2/4\rfloor$; Tur\'{a}n  \cite{Tu} generalized it by determining the exact value of $\ex(n,K_r)$, where $K_r$ denotes a complete graph on $r$ vertices. Erd\H{o}s, Stone and Simonovits \cite{Erdos1966,1946ESS1} found a  connection between $\mathrm{ex}(n, H)$ and the chromatic number $\chi(H)$ of $H$ by proving that 
\begin{align*}
    \ex(n,H)=\left(1-\frac{1}{\chi(H)-1}\right)\frac{n^2}{2}+o(n^2).
\end{align*}
For more results, the reader is referred to the book of Bollob\'{a}s \cite{EGT}.

\subsection{Minor and subdivision} 
Given a graph $H$, a {\textit{subdivision}} of $H$, denoted by $TH$, is a graph obtained by replacing some edges of $H$ by internally vertex-disjoint paths.
Define $H$ to be a {\textit{minor}} of $G$ if $H$ can be obtained from $G$ by means of a sequence of vertex deletions, edge deletions, and edge contractions.
Given a family of graphs $\mathbb{H}$, a graph $G$ is called \textit{$\mathbb{H}$-subdivision-free} if it contains no subdivision of any graph $H\in \mathbb{H}$ as a subgraph, and \textit{$\mathbb{H}$-minor-free} if it does not contain any graph in $\mathbb{H}$ as a minor. We say $G$ is \textit{$\mathbb{H}$-subdivision-saturated} if $G$ is $\mathbb{H}$-subdivision-free, but the addition of any new edge in $G$ creates an $H$-subdivision for some $H\in \mathbb{H}$. 

In 1930, Kuratowski~\cite{kuratowski1930probleme} proved that a graph is planar if and only if it contains no subdivision of $K_5$ or $K_{3,3}$.
Later in 1937, Wagner~\cite{Wagner1937}, in another way, proved that a graph is planar if and only if it contains no minor of $K_5$ or $K_{3,3}$.
As a generalization of planar graphs, it is natural to study problems on $K_s$-subdivision(minor)-free or $K_{s,t}$-subdivision(minor)-free graphs. 
The celebrated  Hadwiger conjecture~\cite{1943Hadwiger} states that 
every $K_r$-minor-free graph is $(r-1)$-colorable for every integer $r\geq1$. In the 1980s, Kostochka~\cite{kostochka1982minimum} and Thomason~\cite{thomason1984extremal} independently proved that for large $r$, every $K_r$-minor-free graph contains at most $\Theta(r\sqrt{\log rn})$ edges and hence it is $O(r\sqrt{\log r})$-colorable. Recently, Alon, Krivelevich and Sudakov \cite{Alon2023} provided a short and self-contained proof of the celebrated Kostochka-Thomason bound. 
Norin, Postle and Song \cite{Song2023} showed that every $K_r$ minor-free graph 
is $O(r(\log r)^\beta)$-colorable for every $\beta>\frac{1}{4}$. 

Observe that any $\mathbb{H}$-minor-free graph is necessarily $\mathbb{H}$-subdivision-free. This naturally leads to the fundamental question: Under what conditions must a graph contain an $\mathbb{H}$-subdivision? In 1961, Haj\'{o}s \cite{hajos1961uber} conjectured a strengthening of Hadwiger's conjecture that every graph $G$ with chromatic number $\chi(G)\geq t$ contains a $TK_t$. Dirac \cite{Dirac1952} showed that this conjecture is true for $t\leq 4$, but Catlin \cite{Catlin1979}  disproved the conjecture for all $t\geq 7$.
 
As a stronger and more fundamental question, conditions on average degree
 guaranteeing an $H$-subdivision have been extensively studied. Bollob\'{a}s and Thomason 
\cite{Bollobas} proved a nice structural result that highly connected graphs are highly linked. Their result,
together with Mader's result \cite{Mader1967} on subgraphs with high connectivity, we obtain the following. 
\begin{theorem}[\cite{Bollobas,Mader1967}]\label{lem2.1}
    Let $G$ be a graph of order $n$ without isolated vertices, and let $H$ be any graph. If $G$ is $H$-subdivision-free, then $e(G)\leq 100e(H)n$. 
\end{theorem}
 Liu and Montgomery~\cite{liu2017proof} proved that 
there exists a constant  $c=c(s,t)>0$ such that every $K_{s,t}$-free ($t\geq s\geq 2$) graph with average degree $d$ contains a $TK_{cd^{\frac{s}{2(s+1)}}}$.  A \textit{balanced subdivision}  of $H$ is obtained by replacing all edges of $H$ with  internally disjoint paths of the same length. Recently, Kim, Liu, Tang, Wang, Yang and Yang~\cite{yang-sub} proved that for any graph $H$, a
 linear-in-$e(H)$ bound on average degree guarantees a balanced $H$-subdivision. 

\subsection{Spectral extremal problems}
The spectral extremal problem is another interesting direction of the Tur{\'a}n-type problems. The spectral Tur\'{a}n type problem (also known as Brualdi-Solheid-Tur\'{a}n type problem, see \cite{Nik2}) states that: For a graph family $\mathbb{H}$, what is the maximal spectral radius of an $\mathbb{H}$-free graph with order $n$? Over the past decade, much attention has been paid to the Brualdi-Solheid-Tur\'{a}n type problem. For more details, one may consult the references, such as for $\mathbb{H}= \{K_r\}$ \cite{Nik3,Wilf}, $\mathbb{H}= \{K_{s,t}\}$ \cite{Bab,Nik3,Nik4}, $\mathbb{H}= \{C_4\}$ \cite{Nik6,ZW},  $\mathbb{H}= \{C_6\}$ \cite{zhai1}, $\mathbb{H}= \{C_3,C_5,\ldots,C_{2k+1}\}$ \cite{Sun2022}, and surveys \cite{surveyli,Liuning}.

For a family of graphs $\mathbb{H}$, it is natural to consider what the maximum spectral radius is of an $\mathbb{H}$-minor ($\mathbb{H}$-subdivision) free graph with  order $n$?  Let ${\rm spex}(n,\mathbb{H}_{\rm minor})$ (resp. ${\rm spex}(n,\mathbb{H}_{\rm sub})$) be the maximum spectral radius over all $n$-vertex $\mathbb{H}$-minor (resp. $\mathbb{H}$-subdivision) free graphs, and let ${\rm SPEX}(n,\mathbb{H}_{\rm minor})$ (resp. ${\rm SPEX}(n,\mathbb{H}_{\rm sub})$) be the family of extremal graphs achieving this maximum.  
This problem has stimulated significant research interest for numerous well-studied graph families 
(see for example, \cite{Hong,Nikiforov-minor,Tait2019,Tait2017,Zhai-minor}). Zhai, Fang, and Lin~\cite{zhai} established a unified solution to this problem for $\mathbb{H}$-minor-free graphs, which can be formally stated as follows. 

A \textit{generalized book}, denoted by~$B_{s,t}$, is obtained by joining a $s$-clique with an independent set of $t$ vertices, that is, $B_{s,t}:=K_s\vee (tK_1)$. Let  $\alpha_H$ be the independence number of $H$. Define $\gamma_H:=|H|-\alpha_H-1$ and $\gamma_\mathbb{H}:=\min_{H\in\mathbb{H}}\gamma_H$.
\begin{theorem}[\cite{zhai}]\label{thm1.2}
    Let $\mathbb{H}$ be a family of graphs with $\gamma_{\mathbb{H}}\geq 1$ and $n$ be a sufficiently large integer. Then  every graph in ${\rm SPEX}(n,\mathbb{H}_{\rm minor})$ contains a spanning subgraph $B_{\gamma_\mathbb{H},n-\gamma_\mathbb{H}}$.
\end{theorem}

Recently, Byrne, Desai, and Tait \cite{byrne2024general} proved a general spectral extremal result that characterizes the maximum spectral radius problem for an extensive class of forbidden graph families. 
As an application of their result, they showed that for a graph $H$, if $B_{k,n-k}$ is $H$-subdivision-saturated,  then ${\rm SPEX}(n, H_{\rm sub}) = B_{k,n-k}$.
In this paper, we extend this characterization and develop a spectral framework for subdivision problems analogous to Theorem~\ref{thm1.2} in the minor-free setting. 

\begin{theorem}\label{thm1.3}
    Let $\mathbb{H}$ be a family of graphs with $\gamma_{\mathbb{H}}\geq 1$ and $n$ be a sufficiently large integer. Then  every graph in ${\rm SPEX}(n,\mathbb{H}_{\rm sub})$ contains a spanning subgraph isomorphic to $B_{\gamma_\mathbb{H},n-\gamma_\mathbb{H}}$.
\end{theorem}
Theorem~\ref{thm1.3} directly implies the following corollary, which  simultaneously extending the applicability of a result by Byrne, Desai, and Tait \cite{byrne2024general}. 
\begin{corollary}\label{cor1}
    Let $\mathbb{H}$ be a family of graphs with $\gamma_{\mathbb{H}}\geq 1$ and $n$ be a sufficiently large integer.  If $B_{\gamma_\mathbb{H},n-\gamma_\mathbb{H}}$ is $\mathbb{H}$-subdivision-saturated,  then ${\rm SPEX}(n, \mathbb{H}_{\rm sub}) = B_{\gamma_\mathbb{H},n-\gamma_\mathbb{H}}$.
\end{corollary}
\begin{example}
    {\rm Let $\mathbb{H}=\{K_{r_1},\ldots, K_{r_s},C_{2\ell_1+1},\ldots, C_{2\ell_t+1},P_{2k_1},\ldots,P_{2k_q}\}$ be a set consisting of complete graphs, odd cycles and paths with even orders, where $3\leq r_1\leq \cdots\leq r_s$, $1\leq \ell_1\leq\cdots\leq \ell_t$ and $2\leq k_1\leq \cdots\leq k_q$. Then $\gamma_{\mathbb{H}}=\min\{r_1-2,\ell_1,k_1-1\}$. Notice that $B_{r_1-2,n-r_1+2}$ is $K_{r_1}$-subdivision-saturated, $B_{\ell_1,n-\ell_1}$ is $C_{2\ell_1+1}$-subdivision-saturated, and $B_{k_1-1,n-k_1+1}$ is $P_{2k_1}$-subdivision-saturated. Therefore,  $B_{\gamma_\mathbb{H},n-\gamma_\mathbb{H}}$ is $\mathbb{H}$-subdivision-saturated. Together with Corollary~\ref{cor1}, we obtain that ${\rm SPEX}(n, \mathbb{H}_{\rm sub}) = B_{\gamma_\mathbb{H},n-\gamma_\mathbb{H}}$.}
\end{example}
\subsection{Proof overviews and organization}

We prove Theorem \ref{thm1.3} by combining the following four techniques: graph transformations, the stability method, the eigenvector method and the double-eigenvector method. For a given graph family $\mathbb{H}$ with $\gamma_{\mathbb{H}}\geq 1$, choose an arbitrary graph $G^* \in {\rm SPEX}(n,\mathbb{H}_{\rm sub})$.  We  divide our proof into three steps. 

\begin{itemize}
    \item Step 1. Following the standard stability approach (Lemma \ref{lem2.4}), we first show that all $G^*$ must be structurally close (in edit distance) to our candidate extremal configuration 
(i.e., the generalized book). That is,  $G^*$ contains a set $L$ of exactly $\gamma_\mathbb{H}$ vertices such that each of which has degree close to $n-1$. The proof of Lemma~\ref{lem2.4} adapts the strategy from \cite{zhai} and the proof framework of \cite{byrne2024general} (see Section 4). 

\item Step 2. We partition the vertex set $V(G^*)\setminus L$ as $S'\cup S''$, where $S''=\{v:L\subseteq N_{G^*}(v)\}$. Our first goal is to show $S'=\emptyset$ (Lemma \ref{lem3.6}). Suppose not, we employ the following arguments. 
\begin{itemize}
    \item Prove $G^*[S'']$ contains a long path $P^0$ (Lemma \ref{lem3.3}). 
    \item By applying the two graph transformations in Figure \ref{fig:1}, we obtain  a graph $G$ containing a linear path with length at least $\max_{H\in \mathbb{H}} \{2|H|^2\}$, which has spectral radius greater than $\rho(G^*)$. It follows from the choice of $G^*$ that $G$ contains an $H$-subdivision for some $H\in \mathbb{H}$.
    \item Choose a minimal $H$-subdivision in $G$, say $H_S$, such that it contains as many vertices as possible in $S''$. Then $S''\subseteq V(H_S)$. However,  each subdivided path in $H_S$ contains at most two vertices from the linear path of $G$, which contradicts the required length of the linear path  and thus proves $S' = \emptyset$. 
\end{itemize}
\item Step 3. We finish the proof by showing $G^*[L]$ is complete through a graph transformation and an argument analogous to that in Step 2. 
\end{itemize}
 
The remainder of the paper is structured as follows.
Section~\ref{sec2} introduces necessary notation and preliminaries. 
In Section~\ref{sec3} we prove Theorem~\ref{thm1.3} using the stability method (Lemma~\ref{lem2.4}). The proof of Lemma~\ref{lem2.4} appears in Section~\ref{sec4}, and  Section~\ref{sec5} provides further characterization of graphs in ${\rm SPEX}(n,\mathbb{H}_{\rm sub})$. In Section 6, we give some concluding remarks. 

\section{Notation and preliminaries}\label{sec2}

\subsection{Notation}

Given a graph $G=(V(G),E(G))$, we always use $n:=|G|$ to denote the order of $G$. Its \textit{adjacency matrix} $A(G)$ is an {$n\times n$\, $0$-$1$ matrix} whose $(i,j)$-entry is $1$ if and only if $ij\in E(G)$.  For a connected graph $G$, it is obvious that $A(G)$ is a real symmetric, nonnegative and irreducible matrix. {Hence,} its eigenvalues are real and can be given in non-increasing order as $\rho_1(G)\geqslant \rho_2(G)\geqslant \cdots \geqslant \rho_n(G)$. The largest modulus of an eigenvalue of $A(G),$  denoted by $\rho(G),$ is called {the} \textit{spectral radius} of $G$. In the sequel, we will usually suppress the graph name from our notation. By the famous Perron-Frobenius Theorem, there exists a non-negative unit eigenvector, say ${\bf x}=(x_1,x_2,\ldots,x_n)^T,$ of $A(G)$ corresponding to  $\rho(G)$. For a vector ${\bf x}$ on vertices of $G,$ we denote by $x_u$ the entry of ${\bf x}$ at $u\in V(G).$

Let $G$ be a graph  with two disjoint vertex subsets $S, T\subseteq V(G)$. Denote by $G[S]$ the subgraph of $G$ induced by $S$ {and $G-S$ the subgraph of $G$ induced by $V(G)\setminus S$. For singletons, we write $G-u$ instead of $G-\{u\}$.} 
Let $G[S,T]$ be the bipartite subgraph obtained from $G[S\cup T]$ by deleting all its edges within $S$ or $T$.
We use $e(S)$ and $e(S,T)$ to denote the numbers of edges in $G[S]$ and $G[S,T]$, respectively.
For a graph $G=(V,E)$ and a subset $E_1\subseteq \binom{V}{2}$, we define $G\pm E_1=(V,E\pm E_1)$ to be a simple  graph obtained from 
$G$ by adding or deleting edges in $E_1$. For $u\in V(G)$, let $N_G(u)$ denote its neighborhood in $G$. We write $N_S(u):=N_G(u)\cap S$ and $d_S(u):=|N_S(u)|$.
Let $G\cup G'$ be the union of two vertex-disjoint graphs $G$ and $G'$.

Let $H$ be a graph with $V(H)=\{v_1,\ldots,v_h\}$ and $E(H)=\{e_1,\ldots,e_{m}\}$. We say $G$ contains a \textit{model} of $H$  if all of the following hold: 

\begin{itemize}
    \item There is a vertex mapping $\phi: V(H) \to V(G)$ where $\phi(v_i) \neq \phi(v_j)$ for $i \neq j$. 
    \item For each edge $e_k = v_iv_j \in E(H)$, there is a path $P^k$ in $G$ connecting $\phi(v_i)$ and $\phi(v_j)$. In particular, $P^k=e_k$ if $e_k \in E(G)$.
    \item All paths ${P^k}$ are internally vertex-disjoint from each other and no internal vertex of any $P^k$
  belongs to $\{\phi(v_i):i\in [h]\}$. 
\end{itemize}
We always call vertices in $\{\phi(v_i):i\in [m]\}$ {\textit{branch vertices}} and internal vertices in each $P^k$ \textit{subdivision vertices}. 
It is not hard to see that $G$  contains an $H$-subdivision if and only if $G$ admits a model of $H$. 
The tuple $(P^1,\ldots,P^m)$ is called an 
$H$-\textit{partition} of the $H$-subdivision in $G$. 
An $H$-subdivision is \textit{minimal} if its corresponding $H$-partition minimizes $\sum_{i\in [m]}|V(P^i)|$ among all possible $H$-partitions in $G$. 

\subsection{Preliminaries}

In this subsection, we establish some preliminaries that will be essential for our proof. 
The subsequent one is obvious.

\begin{lemma}\label{lem2.2}
    Let $H$ be a graph. Then both $B_{\gamma_H+1,\alpha_H}$ and $K_{\gamma_H+1,\alpha_H+\binom{\gamma_H+1}{2}}$ contain $H$-subdivisions.
\end{lemma}
Our next lemma, which follows directly from \cite[Lemma 2.6]{Zhai-minor}, provides a lower bound  for the maximum spectral radius of $\mathbb{H}$-subdivision-free graphs. 

\begin{lemma}\label{lem2.3}
Let $\mathbb{H}$ be a family of graphs. Then $B_{\gamma_\mathbb{H},n-\gamma_\mathbb{H}}$ is $\mathbb{H}$-subdivision-free and ${\rm spex}(n,\mathbb{H}_{\rm sub})\geq \sqrt{\gamma_\mathbb{H}(n-\gamma_\mathbb{H})}$.
\end{lemma}

In the following paper, for a given graph family $\mathbb{H}$, we define the constant $C_\mathbb{H}$ as
$$
C_\mathbb{H}>\max_{H\in \mathbb{H}}\left\{2|H|^3,100\cdot e(H)\right\}.
$$
Following the standard stability approach, we initiate our analysis by characterizing the approximate local structure of graphs in $\mathrm{SPEX}(n,\mathbb{H})$. This establishes that all near-extremal graphs must be structurally close (in edit distance) to the generalized book.
Our proof strategy parallels that of Zhai, Fang, and Lin \cite{zhai}, while incorporating necessary modifications for our proof. For completeness and readability, we defer the detailed proof to Section~4.

\begin{lemma}[Stability result]\label{lem2.4}
    Assume that $n\in \mathbb{N}$ is large enough and $\mathbb{H}$ is a family of graphs with $\gamma_{\mathbb{H}}\geq 1$. 
    Let $G$ be an $\mathbb{H}$-subdivision-free graph of order $n$. Let ${\bf x}=(x_1,x_2,\ldots, x_n)^T$ be a non-negative unit eigenvector corresponding to $\rho(G)$ with $x_{u^{*}}=\max_{u \in V(G)}x_{u}$.
    If $\rho(G) \geq \sqrt{\gamma_\mathbb{H}(n-\gamma_\mathbb{H})}$, then $G$ admits a set $L$ of exactly $\gamma_\mathbb{H}$ vertices such that 
    $$
    x_{u} \geq \left(1-\frac{1}{2(10C_\mathbb{H})^{2}}\right)x_{u^{*}}\ \ \text{and}\ \ d_{G}(u) \geq \left(1-\frac{1}{(10C_\mathbb{H})^{2}}\right)n\ \text{for every}\ u \in L.
    $$ 
\end{lemma}

\section{Proof of Theorem \ref{thm1.3}}\label{sec3}
In this section, we give the proof of Theorem \ref{thm1.3} by using the stability result. Assume that  $\mathbb{H}$ is a family of graphs with $\gamma_{\mathbb{H}}\geq 1$. A graph $H$ is said to be \textit{minimal} with respect to $\mathbb{H}$, if $H\in \mathbb{H}$ such that (i) $\gamma_H=\gamma_{\mathbb{H}}$, (ii) subject to (i), $|H|$ is also minimal. Let $H^*$ be an arbitrary minimal graph with respective to $\mathbb{H}$. It is obvious that all minimal graphs have the same independence number. Thus, we can set $\alpha_{\mathbb{H}}:=\alpha_{H^*}$. 

Choose an arbitrary graph $G^* \in {\rm SPEX}(n,\mathbb{H}_{\rm sub})$, where $n$ is sufficiently large. 
Let ${\bf x}=(x_{1},\ldots,x_{n})^T$ be a non-negative unit eigenvector of 
$A(G^*)$  corresponding to $\rho(G^*)$, and $u^{*} \in V(G^*)$ be a vertex such that $x_{u^{*}}=\max_{u \in  V(G^*)}x_{u}$.
Set $\rho^{*}:=\rho(G^*)$.
It follows from Lemma~\ref{lem2.3} that  $\rho^{*}\geq \sqrt{\gamma_\mathbb{H}(n-\gamma_\mathbb{H})}$. Together with Lemma~\ref{lem2.4}, we obtain that $G^*$ contains a set $L$ of $\gamma_\mathbb{H}$ vertices such that for every $u \in L$ one has 
$$x_{u} \geq \left(1-\frac{1}{2(10C_\mathbb{H})^{2}}\right)x_{u^{*}}\ \ \text{and}\ \ d_{G^*}(u) \geq \left(1-\frac{1}{(10C_\mathbb{H})^{2}}\right)n.
$$ 

For convenience, assume that $L:=\{v_1,v_2,\ldots ,v_{\gamma_\mathbb{H}}\}$.   
Let $S=V(G^*)\setminus L$ 
and we further divide $S$ into $S'\cup S''$, where $S''=\{v:L\subseteq N_{G^*}(v)\}$. Our main proof consists of two key steps: (i) $S'=\emptyset$; (ii) $G^*[L]$ induces a complete graph.

Notice that every vertex in $L$ has at most $\frac{n}{(10C_\mathbb{H})^2}$ non-neighbors in $S'$.
Hence, 
\begin{align}\label{eq:1}
    |S'|\leq\frac{n}{(10C_\mathbb{H})^2}|L|=\frac{\gamma_\mathbb{H}n}{(10C_\mathbb{H})^2}\ \ \text{and}\ \ |S''|\geq n-\gamma_\mathbb{H}-\frac{\gamma_\mathbb{H}n}{(10C_\mathbb{H})^2}>\binom{\gamma_\mathbb{H}}{2}+\max_{H\in \mathbb{H}}\{|H|\}.
\end{align}
Our first goal is to prove that $S'=\emptyset$. {For this, we need to estimate the degree of vertices in $S$ and the entry in ${\bf x}$ corresponding to vertices in $S$.} 

\begin{lemma}\label{lem3.1}
      For every $v\in S'$, 
    we have $d_{S''}(v)\leq \max_{H\in \mathbb{H}}\{|H|\}-2$; for every $v\in S''$, we have $d_{S''}(v)\leq\alpha_{\mathbb{H}}$.
\end{lemma}
  \begin{proof}
    We only prove that $d_{S''}(v)\leq \max_{H\in \mathbb{H}}\{|H|\}-2$ for every  $v\in S'$, and the other inequality can be proved similarly.
    Suppose that there exists a vertex $v_0\in S'$ such that $$d_{S''}(v_0)\geq\max_{H\in \mathbb{H}}\{|H|\}-1\geq |H^*|-1=\gamma_{\mathbb{H}}+\alpha_{\mathbb{H}},$$
    where $H^*$ is a minimal graph with respect to $\mathbb{H}$. 
    Let 
    $$
    \{u_1,\ldots, u_{\gamma_{\mathbb{H}}},w_1,\ldots, w_{\alpha_{\mathbb{H}}}\}\subseteq N(v_0)\cap S''.
    $$
    Clearly, $G^*$ contains a set of internally disjoint paths $\{v_0u_iv_i:i\in \{1,\ldots,\gamma_{\mathbb{H}}\}\}$.
    Furthermore, based on~\eqref{eq:1}, for each pair of vertices $v_i,v_j\in L$,  there exists a path $P_{ij}=v_iv_{i,j}v_j$ such that $v_{i,j}\in S''\setminus \{u_1,\ldots, u_{\gamma_{\mathbb{H}}},w_1,\ldots, w_{\alpha_{\mathbb{H}}}\}$ and $v_{i,j}=v_{i',j'}$ if and only if $(i,j)=(i',j')$. Therefore, $G^*[L\cup \{v_0,u_1,\ldots, u_{\gamma_{\mathbb{H}}}\}\cup \{v_{i,j}:1\leq i<j\leq \gamma_{\mathbb{H}}\}]$ forms a subdivision of $K_{\gamma_{\mathbb{H}}+1}$ with branch vertices $\{v_0,\ldots, v_{\gamma_{\mathbb{H}}}\}$. 
    
    Notice that $w_1,\ldots, w_{\alpha_{\mathbb{H}}}$ are adjacent to all vertices in $\{v_0,\ldots, v_{\gamma_{\mathbb{H}}}\}$. It follows that $G^*$ contains a $B_{\gamma_{\mathbb{H}}+1,\alpha_{\mathbb{H}}}$-subdivision (i.e., a $B_{\gamma_{H^*}+1,\alpha_{H^*}}$-subdivision). Together with Lemma~\ref{lem2.2}, $G^*$ contains a subdivision of $H^*$, a contradiction.
  \end{proof}

\begin{lemma}\label{lem3.2}
    For every vertex  $v\in S$ we have $x_v\leq\frac{x_{u^*}}{20C_\mathbb{H}}$.
\end{lemma}
  \begin{proof}
  Let $v$ be an arbitrary vertex in $S$. 
In view of Lemma \ref{lem3.1}, one has 
    \begin{align*}
        d_{G^*}(v)=d_{L\cup S''}(v)+d_{S'}(v)\leq(\gamma_\mathbb{H}+\max_{H\in \mathbb{H}}\{|H|\}-2)+d_{S'}(v)\leq C_\mathbb{H}+d_{S'}(v),
    \end{align*}
    the last inequality holds by the choice of $C_\mathbb{H}$. 
    Therefore,  
    \begin{align}\notag
    \sum_{w\in S'}\rho^*x_w&=\sum_{w\in S'}\sum_{u\in N_{G^*}(w)}x_u\leq \sum_{w\in S'}d_{G^*}(w)x_{u^*}\leq \sum_{w\in S'}\left(C_\mathbb{H}+d_{S'}(v)\right)x_{u^*}\\\label{eq:01}
        &=\left(C_\mathbb{H}|S'|+2e(S')\right)x_{u^*}\leq 3C_\mathbb{H}|S'|x_{u^*}\\\notag
        &\leq \frac{3\gamma_\mathbb{H}n}{100C_\mathbb{H}}x_{u^*},
    \end{align}
    where the inequality in  \eqref{eq:01} holds by Theorem~\ref{lem2.1} and the fact that $G^*[S']$ is $\mathbb{H}$-subdivision-free, and the last inequality holds by \eqref{eq:1}.  
    It follows that  
\begin{align}\label{eq:2}
        \rho^*x_v=\underset{u\in N_{G^*}(v)}{\sum}x_u=\underset{u\in N_{L\cup S''}(v)}{\sum}x_u+\underset{u\in N_{S'}(v)}{\sum}x_u\leq C_\mathbb{H}x_{u^*}+\frac{3\gamma_\mathbb{H}n}{100C_\mathbb{H}\rho^*}x_{u^*}.
    \end{align}
    Dividing both sides of \eqref{eq:2} by $\rho^*$, since ${\rho^*}\geq\sqrt{\gamma_\mathbb{H}(n-\gamma_\mathbb{H})}$ and $n$ is sufficiently large, we obtain $x_v\leq\frac{x_{u*}}{20C_\mathbb{H}}$, as desired.
  \end{proof}
The subsequent lemma guarantees that $G^*[S'']$ contains a long path. 
\begin{lemma}\label{lem3.3}
If $S'\neq \emptyset$, then 
$G^*[S'']$ contains a path of order at least $\max_{H\in \mathbb{H}}\{2|H|^2\}$.
\end{lemma}
  \begin{proof}    
Let $G_0:=G^*-\{uv:u\in S',\, v\in V(G^*)\}+\{uv:u\in S',\,v\in L\}$.
Hence,
    \begin{align}\label{eq:3}
        \rho(G_0)-\rho^*\geq {\bf x}^T\left(A(G_0)-A(G^*)\right){\bf x}\geq\sum_{u\in S'}x_u\left(2\sum_{v\in L}{x_v}-2\sum_{v\in N_{L\cup S''}(u)}x_v-\sum_{v\in N_{S'}(u)}x_v\right).
    \end{align}
Recall that $|L|=\gamma_\mathbb{H}$ and $x_v\geq \left(1-\frac{1}{2(10C_\mathbb{H})^2}\right)x_{u^*}$ for each $v\in L$. Therefore,  
    \begin{align}\label{eq3.1}
        \sum_{v\in L}x_v\geq \gamma_\mathbb{H}\left(1-\frac{1}{2(10C_\mathbb{H})^2}\right)x_{u^*}\geq\left(\gamma_\mathbb{H}-\frac{1}{10}\right)x_{u^*}.
    \end{align}
Furthermore, for each $u\in S'$, we have $d_L(u)\leq|L|-1=\gamma_\mathbb{H}-1$ based on the definition of $S'$, and $d_{S''}(u)\leq \max_{H\in \mathbb{H}}\{|H|\}$ by Lemma \ref{lem3.1}. Together with Lemma \ref{lem3.2}, one has  
\begin{align}\label{eq3.2}
\sum_{v\in N_{L\cup S''}(u)}x_v\leq d_L(u)x_{u^*}+d_{S''}(u)\frac{x_{u^*}}{20C_\mathbb{H}}\leq\left(\gamma_\mathbb{H}-1\right)x_{u^*}+\frac{\max_{H\in \mathbb{H}}\{|H|\}}{20C_\mathbb{H}}x_{u^*}.
\end{align}
    In what follows, we estimate the value of $\sum_{u\in S'}x_u\sum_{v\in N_{S'}(u)}x_v$. We first define $S_i'$ for each $i\in [|S'|]$ iteratively. Let $S_1':=S'$. 
For each $i\geq 2$, choose $u_{i-1}\in S_{i-1}'$ such that $d_{S_{i-1}'}(u_{i-1})=\delta(G^*[S_{i-1}'])$ and let $S_i':=S_{i-1}'\setminus \{u_{i-1}\}$. Notice that $G^*[S_i']$ is $\mathbb{H}$-subdivision-free. Together with Theorem~\ref{lem2.1}, for each $i\in [|S'|]$ one has 
$$
d_{S_i'}(u_i)\leq \frac{2e(S_i')}{|S_i'|}\leq 2C_\mathbb{H}.
$$ 
    Combining Lemma \ref{lem3.2} gives that \allowdisplaybreaks
    \begin{align}\label{eq3.3}
\sum_{u\in S'}x_u\sum_{v\in N_{S'}(u)}x_v=\sum_{i=1}^{|S'|}2x_{u_i}\sum_{v\in N_{S_i'} (u_i)}x_v\leq 
\sum_{i=1}^{|S'|}2x_{u_i}\left(d_{S_i'}(u_i)\frac{x_{u^*}}{20C_\mathbb{H}}\right)
\leq \sum_{i=1}^{|S'|}2x_{u_i}\cdot\frac{x_{u^*}}{10}. 
    \end{align}
Combining \eqref{eq:3}-\eqref{eq3.3} yields that $\rho(G_0)>\rho^*$. 
    
    Based on the choice of $G^*$, we know that $G_0$ contains an $H$-subdivision for some $H\in\mathbb{H}$. Choose a minimal $H$-subdivision in $G_0$, say $H_S$, such that it contains as many vertices as possible in $S''$. Assume that $(P^1,\ldots,P^{m})$ is an $H$-partition of $H_S$ in $G_0$, where $m=e(H)$. 
    
    We claim that $V(H_S)\cap S'\neq\emptyset$ and $S''\subseteq V(H_S)$.
    In fact, if $V(H_S)\cap S'=\emptyset$, it infers that $G^*$ contains an $H_0$-subdivision, a contradiction. 
    In addition, if there exists a vertex $v\in S''\setminus V(H_S)$, then we can use $v$ to replace a vertex $u\in V(H_S)\cap S'$ (since $V(H_S)\cap S'\neq\emptyset$ and $N_{G_0}(u)\subseteq N_{G_0}(v)$), which contradicts the choice of $H_S$. 

    Recall that 
    $H_S$ is a minimal $H$-subdivision in 
$G_0$ and each vertex in $L$ is adjacent to all vertices in $S$ (in $G_0$). Hence, for each $i\in [m]$, 
    \begin{itemize}
        \item if the two endpoints of $P^i$ are in $L$, then  $|V(P^i)\cap S|\in \{0,1\}$;
        \item if the two endpoints  of $P^i$ are in $S$, then it contains at most two vertices in $L$;
        \item if one the endpoints  of $P^i$ lies  in $S$ and the other in $L$, then $|V(P^i)|=2$. 
    \end{itemize}
Therefore, each $P^i-L$ has at most three connected components.  
    Together with $S''\subseteq V(H_S)$ and the fact that $n$ is sufficiently large, there exists a path of order at least $\frac{|S''|}{3\binom{|H|}{2}}\geq \max_{H\in \mathbb{H}}\{2|H|^2\}$ in $G^*[S'']$, as desired.
  \end{proof}

We will maintain the notation of $S_i'$ defined in the proof of Lemma \ref{lem3.3}, that is, we assume  $S':=\{u_1,\ldots,u_{|S'|}\}$,  $S_1':=S'$ and $S_i':=S_{i-1}'\setminus \{u_{i-1}\}$ for every $i\in [2,|S'|]$. Moreover, $d_{S_{i-1}'}(u_{i-1})=\delta(G^*[S_{i-1}'])\leq 2C_{\mathbb{H}}$ for each $i\in [|S'|]$. 
Now, we are to prove  $S'=\emptyset$. 
\begin{lemma}\label{lem3.6}
    $S'=\emptyset$. 
\end{lemma}
\begin{proof}
    Suppose that $S'\neq \emptyset$.  
For each $i\in [\gamma_{\mathbb{H}}]$, denote $L_i:=\{v\in L:u_iv\notin E(G^*)\}$.
Assume, without loss of generality, that $x_{u_k}=\min_{i\in [|S'|]}x_{u_i}$ and choose $v_{j_k}\in L_k$. 
Let $G_1:=G^*-\{uw:u\in S',w\in S\}+\{u_iv:i\in [|S'|],v\in L_i\}-\{u_kv_{j_k}\}$  (see Figure \ref{fig:1}). The following claim  compares  the spectral radii of $G_1$ and $G^*$.

\begin{figure}[ht!]
    \centering
    \begin{tikzpicture}[
    set/.style={ellipse, draw, minimum width=1.5cm, minimum height=1cm, opacity=0.7},
    node distance=1cm,
    point/.style={circle, fill=black, inner sep=1pt},
    label/.style={font=\small}
]

\node[set] (L) at (0,2) {$L$};

\node[point, label=left:$v_{j_k}$] (vjk) at (-0.65,1.8) {};
\node[right=-0.15cm of vjk, align=center] {$v_{j_k}$};
\node[set] (S1) at (-1.2,0) {$S'$};
\node[point, label=left:$u_k$] (uk) at (-1.5,0.35) {};
\node[right=-0.13cm of uk, align=center] {$u_k$};

 \node[set, minimum width=2.2cm, minimum height=1.3cm] (S2) at (1.5,0) {};
\node[above=0cm of S2, align=center] {$S''$};

\node[point] (A1) at (0.5,0) {};
\node[point] (A2) at (1.0,0) {};
\node[point] (A3) at (1.5,0) {};
\node[point] (A4) at (2.5,0) {};
\draw[dashed, black!70,thick] (A3.east) -- (A4.west);
\draw (A1) -- (A2)-- (A3);
\node[point] (A5) at (0.9,-0.25) {};
\node[point] (A6) at (1.1,-0.25) {};
\node[point] (A7) at (1.4,-0.25) {};
\node[point] (A8) at (1.6,-0.25) {};
\draw (A5) -- (A2)-- (A6);
\draw (A7) -- (A3)-- (A8);
\node at (1.5,0.3) {$P^0$};

\foreach \i in {1,...,2} {
    \draw[black!70, ultra thick] (L.south east) -- (S2.north west);
}

\foreach \i in {1,...,2} {
    \draw[dashed, black!50, ultra thick] (L.south west) -- (S1.north east);
}

\draw[dashed, red, thick] (vjk) to node[label, sloped, above] {\footnotesize missing} (uk);

\draw[dashed, black!50] (S1.east) -- (S2.west);
\node at (0,-1) {$G^*$};

\node[set] (L) at (5,2) {$L$};
\node[point, label=left:$v_{j_k}$] (vjk) at (4.35,1.8) {};
\node[right=-0.15cm of vjk, align=center] {$v_{j_k}$};
\node[set] (S1) at (3.8,0) {$S'$};
\node[point, label=left:$u_k$] (uk) at (3.5,0.35) {};
\node[right=-0.13cm of uk, align=center] {$u_k$};
\node[set, minimum width=2.2cm, minimum height=1.3cm] (S2) at  (6.5,0) {};
\node[above=0cm of S2, align=center] {$S''$};
\node[point] (A1) at (5.5,0) {};
\node[point] (A2) at (6.0,0) {};
\node[point] (A3) at (6.5,0) {};
\node[point] (A4) at (7.5,0) {};
\draw[dashed, black!70,thick] (A3.east) -- (A4.west);
\draw (A1) -- (A2)-- (A3);
\node[point] (A5) at (5.9,-0.25) {};
\node[point] (A6) at (6.1,-0.25) {};
\node[point] (A7) at (6.4,-0.25) {};
\node[point] (A8) at (6.6,-0.25) {};
\draw (A5) -- (A2)-- (A6);
\draw (A7) -- (A3)-- (A8);
\node at (6.5,0.3) {$P^0$};

\foreach \i in {1,...,2} {
    \draw[black!70, ultra thick] (L.south east) -- (S2.north west);
}

\foreach \i in {1,...,2} {
    \draw[black!70, ultra thick] (L.south west) -- (S1.north east);
}

\draw[dashed, red, thick] (vjk) to node[label, sloped, above] {\footnotesize missing} (uk);

\node at (5,-1) {$G_1$};

\node[set] (L) at (10,2) {$L$};
\node[point, label=left:$v_{j_k}$] (vjk) at (9.35,1.8) {};
\node[right=-0.15cm of vjk, align=center] {$v_{j_k}$};
\node[set] (S1) at (8.8,0) {$S'$};
\node[point, label=left:$u_k$] (uk) at (8.5,0.35) {};
\node[right=-0.13cm of uk, align=center] {$u_k$};

\node[set, minimum width=2.2cm, minimum height=1.3cm] (S2) at (11.5,0) {};
\node[above=0cm of S2, align=center] {$S''$};

\node[point] (A1) at (10.5,0) {};
\node[point] (A2) at (11.0,0) {};
\node[point] (A3) at (11.5,0) {};
\node[point] (A4) at (12.5,0) {};
\draw[dashed, black!70,thick] (A3.east) -- (A4.west);
\draw (A1) -- (A2)-- (A3);
\foreach \i in {1,...,2} {
    \draw[black!70, ultra thick] (L.south east) -- (S2.north west);
}

\foreach \i in {1,...,2} {
    \draw[black!70, ultra thick] (L.south west) -- (S1.north east);
}

\draw[red](vjk) -- (uk);
\node at (10,-1) {$G_2$};

\end{tikzpicture}
    \caption{Graph transformations.}
    \label{fig:1}
\end{figure}
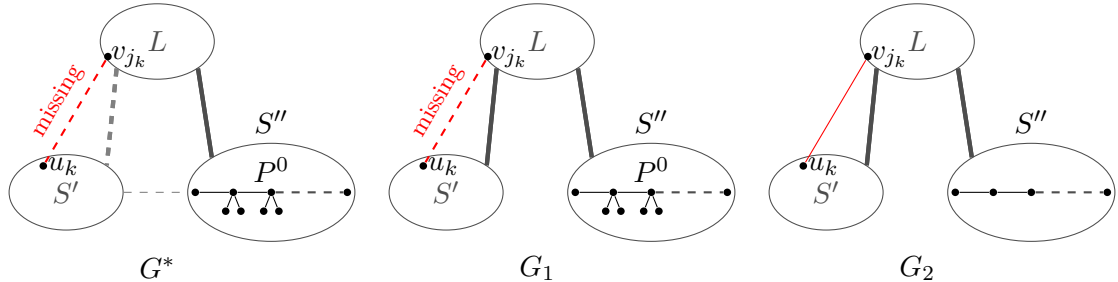
\begin{claim}\label{lem3.4}
If $e(L,S')\neq |L||S'|-1$, then $\rho(G_1)>\rho^*$.
\end{claim}
  \begin{proof}[\bf Proof of Claim \ref{lem3.4}]
Assume that $e(L,S')\neq |L||S'|-1$, that is,  $G^*[L,S']\neq G_1[L,S']$. Hence, either $|S'|\geq 2$ or $|S'|=1$ and $d_L(u_1)\leq\gamma_\mathbb{H}-2$.
Let $k'$ be an integer in $[|S'|]\setminus \{k\}$ if $|S'|\geq 2$, and let $k'=1$ otherwise. Recall that $x_v\geq (1-\frac{1}{(10C_\mathbb{H})^2})x_{u^*}\geq\frac{1}{2}x_{u^*}$ for each $v\in L$.  It is routine to check that \allowdisplaybreaks
    \begin{align*}\notag
        \rho(G_1)-\rho^* \geq& {\bf x}^T\left(A(G_1)-A(G^*)\right){\bf x}\\
        =&\sum_{i=1}^{|S'|}2x_{u_i}\left(\sum_{v\in L_i}x_v-\sum_{v\in N_{S'_{i}}(u_i)}x_v-\sum_{v\in N_{S''}(u_i)}x_v\right)-2x_{u_k}x_{v_{j_k}}\\
        \geq&\sum_{i\in [|S'|]\setminus \{k,k'\}}2x_{u_i}\left(\sum_{v\in L_i}x_v-\sum_{v\in N_{S'_{i}}(u_i)}x_v-\sum_{v\in N_{S''}(u_i)}x_v\right)\\
        &+2x_{u_{k'}}\left(\frac{1}{2}x_{u^*}-\sum_{v\in  N_{S'_{k'}}(u_{k'})}x_v-\sum_{v\in  N_{S''}(u_{k'})}x_v-\sum_{v\in N_{S'_{k}}(u_k)}x_v-\sum_{v\in N_{S''}(u_k)}x_v\right).
    \end{align*}

Together with Lemmas~\ref{lem3.1}-\ref{lem3.2}, for each $i\in [|S'|]\setminus \{k,k'\}$ one has 
\begin{align*}
    \sum_{v\in L_i}x_v-\sum_{v\in N_{S'_{i}}(u_i)}x_v-\sum_{v\in N_{S''}(u_i)}x_v
    \geq\frac{1}{2}x_{u^*}-2C_\mathbb{H}\frac{x_{u^*}}{20C_\mathbb{H}}-\max_{H\in \mathbb{H}}\{|H|\}\frac{x_{u^*}}{20C_\mathbb{H}}
    >0,
\end{align*}
and similarly 
\begin{align*}
   &\frac{1}{2}x_{u^*}-\sum_{v\in  N_{S'_{k'}}(u_{k'})}x_v-\sum_{v\in  N_{S''}(u_{k'})}x_v-\sum_{v\in N_{S'_{k}}(u_k)}x_v-\sum_{v\in N_{S''}(u_k)}x_v>0.
\end{align*}
If $\max\{x_{u_i}:i\in [|S'|]\}>0$, then $\rho(G_1)>\rho^*$, as desired. If $x_{u_1}=\cdots=x_{u_{|S'|}}=0$, then each vertex in $S'$ is not adjacent to any vertex in $L\cup S''$, that is, $G^*[L\cup S'']$ is a connected component of $G^*$. Furthermore, $\rho(G^*[L\cup S''])=\rho^*$.  However, $G^*[L\cup S'']$ is a proper subgraph of some connected component of $G_1$. Hence, $\rho(G_1)>\rho^*$, as desired.
 \end{proof}

A \textit{linear path} in a graph $G$ is a path where every internal vertex along the path has degree exactly $2$ in $G$, while the endpoints may  coincide (forming a cycle). 
Denote $\zeta:=\max_{H\in \mathbb{H}}\{2|H|^2\}$. Based on Lemma~\ref{lem3.3}, we know that $S''$ contains a path $P$ of  order at least $\zeta$. 
Let $P^0=w_1w_2\ldots w_{\zeta}$ be a subpath of $P$. 
Notice that $P^0$ may not be a linear path.
Let $G_2$ be the graph obtained from $G_1+\{u_kv_{j_k}\}$ by deleting all edges within $S''$ that are incident to some vertices in $P^0$ but do not belong to $E(P^0)$  (see Figure \ref{fig:1}). 
Next, we compare the spectral radii of $G_2$ and $G^*$. 

\begin{claim}\label{lem3.5}
 $\rho(G_2)>\rho^*$. 
\end{claim}
\begin{proof}[\bf Proof of Claim \ref{lem3.5}]
    Let ${\bf y}=(y_1,y_2,\ldots,y_n)^{T}$ be a non-negative unit eigenvector of 
$A(G_1)$ with respect to $\rho(G_1)$, and let ${\bf z}=(z_1,z_2,\ldots, z_n)^T$ be a non-negative unit eigenvector of $A(G_2)$ with respect to $\rho(G_2)$. Let 
    $y_{S''}^*:=\max_{w\in S''}y_w$,  $\sigma_L:=\sum_{u\in L}y_u$, and let $z_{S''}^*=\max_{u\in S''}z_u$. 
We proceed by  considering the following two cases.

{\bf Case 1.} $e_{G^*}(L,S')\neq |L||S'|-1$.
    
    In view of Claim~\ref{lem3.4}, we have $\rho(G_1)>\rho^*$, so it suffices to prove $\rho(G_2)>\rho(G_1)$. Notice that $\rho(G_1)y_{S''}^*\leq \alpha_{\mathbb{H}}y_{S''}^*+\sigma_L$, this implies 
    \begin{align}\label{ys}
        y_{S''}^*\leq \frac{\sigma_L}{\rho(G_1)-\alpha_{\mathbb{H}}}.
    \end{align}

By the construction of
$G_1$, we have $|E(G_1)|\leq |E(G^*)|+|L||S'|\leq C_\mathbb{H} n+\frac{\gamma_\mathbb{H}^2n}{(10C_\mathbb{H})^2}$. Together with the well-known fact that 
$\rho(G_1)\leq \sqrt{2|E(G_1)|}$, we have 
$$
\sqrt{2\left(C_\mathbb{H} n+\frac{\gamma_\mathbb{H}^2n}{(10C_\mathbb{H})^2}\right)}\rho(G_1)y_{v_{j_k}}\geq \rho(G_1)^2y_{v_{j_k}}\geq \sum_{u\in S\setminus \{u_k\}}\rho(G_1)y_u\geq (|S|-1)\sigma_L=(n-\gamma_\mathbb{H}-1)\sigma_L.
$$
It follows that 
\begin{align}\notag
    y_{v_{j_k}}&\geq \frac{(n-\gamma_\mathbb{H}-1)\sigma_L}{\rho(G_1)\sqrt{2\left(C_\mathbb{H} n+\frac{\gamma_\mathbb{H}^2n}{(10C_\mathbb{H})^2}\right)}}\\\notag
    &\geq \frac{\sigma_L}{\rho(G_1)-\alpha_{\mathbb{H}}}\cdot\frac{n-\gamma_\mathbb{H}-1}{2\sqrt{C_\mathbb{H} n+\frac{\gamma_\mathbb{H}^2n}{(10C_\mathbb{H})^2}}}\\\label{eq:4.1}
    &>\frac{\sigma_L\cdot 4\zeta\alpha_{\mathbb{H}}}{\rho(G_1)-\alpha_{\mathbb{H}}}\geq 4\zeta\alpha_{\mathbb{H}}y_{S''}^*.
\end{align}

    We first assume that $\gamma_\mathbb{H}\geq2$. Clearly, 
    \begin{align}\notag\label{eq:4}
        \rho(G_2)-\rho(G_1)&\geq{\bf y}^T\left(A(G_2)-A(G_1)\right){\bf y}\\&\geq 2y_{v_{j_k}}y_{u_k}-2\sum_{i=1}^{\zeta}\sum_{w\in N_{S''}(w_i)}y_{w_i}y_w.
    \end{align}
We claim that $\sigma_L\geq 2y_{v_{j_k}}$. In fact, if $y_{v_{j_k}}\leq y_{v_{\ell}}$ for some $\ell\neq j_k$, then $\sigma_L\geq 2y_{v_{j_k}}$. If $y_{v_{j_k}}> y_{v_{\ell}}$ for every  $\ell\neq j_k$, then $N_L(v_{j_k})\neq \emptyset$ and  
$$
\rho(G_1)\left(\sum_{v_\ell\in N_L(v_{j_k})}y_{v_\ell}-y_{v_{j_k}}\right)\geq |N_L(v_{j_k})|y_{v_{j_k}}-\sum_{v_\ell\in N_L(v_{j_k})}y_{v_\ell}>0.
$$
Therefore, $\sigma_L\geq 2y_{v_{j_k}}$, as desired. 
Notice that $\rho(G_1)y_{u_k}=\sigma_L-y_{v_{j_k}}$, 
which combined with \eqref{ys}, implies that  
\begin{align*}
y_{S''}^*\leq\frac{\sigma_L}{\rho(G_1)-\alpha_\mathbb{H}}\leq\frac{2(\sigma_L-y_{v_{j_k}})}{0.5\rho(G_1)}=4y_{u_k}.
\end{align*}
   Therefore, 
    \begin{align*}
        \sum_{i=1}^{\zeta}\sum_{w\in N_{S''}(w_i)}y_{w_i}y_w\leq\zeta\alpha_\mathbb{H}\cdot y_{S''}^*\cdot4y_{u_k}=4\zeta\alpha_\mathbb{H}\cdot y_{S''}^*y_{u_k}.
    \end{align*}
Combining this with \eqref{eq:4.1} and \eqref{eq:4}, we get that $\rho(G_2)-\rho(G_1)>0$, as desired. 
    
    Next, we assume that $\gamma_\mathbb{H}=1$ and $v_1 (=v_{j_k})$ is the unique vertex in $L$.
    Notice that $\rho(G_2)z_{S''}^*\leq z_{v_1}+\alpha_\mathbb{H}z_{S''}^*$ 
    and $\rho(G_2)z_{u_k}=z_{v_1}$. 
    This yields $(\rho(G_2)-\alpha_\mathbb{H})z_{S''}^*\leq \rho(G_2)z_{u_k}$, 
    which implies $z_{S''}^*\leq 2z_{u_k}$. Combining this with \eqref{eq:4.1}, one has 
    \begin{align*}
        {\bf y}^T{\bf z}(\rho(G_2)-\rho(G_1))&={\bf y}^T\left(A(G_2)-A(G_1)\right){\bf z}\\
        &=y_{u_k}z_{v_1}+y_{v_1}z_{u_k}-\sum_{i=1}^{\zeta}\sum_{w\in N_{S''}(w_i)}(y_{w_i}z_w+y_wz_{w_i})\\&\geq y_{v_1}z_{u_k}-2\zeta\cdot\alpha_\mathbb{H}\cdot y_{S''}^*z_{S''}^*\\&>0.
    \end{align*}
That is, $\rho(G_2)>\rho(G_1)$, as desired. 

{\bf Case 2.} $e_{G^*}(L,S')=|L||S'|-1$.

In this case, $S'=\{u_1\}$ and $d_L(u_1)=\gamma_\mathbb{H}-1$. 
    Notice that $\rho(G_2)z_{u_1}=\sum_{w\in L}z_w$ and $\rho(G_2)z_{S''}^*\leq\sum_{w\in L}z_w+\alpha_\mathbb{H}z_{S''}^*$, 
    which implies $z_{S''}^*\leq 2z_{u_1}$. Recall that $x_{v_{j_1}} \geq \left(1-\frac{1}{2(10C_\mathbb{H})^{2}}\right)x_{u^{*}}$.  Together with Lemmas~\ref{lem3.1}-\ref{lem3.2} and the choice of $C_\mathbb{H}$, one has
    \begin{align*}
        {\bf x}^T{\bf z}(\rho(G_2)-\rho^*)&={\bf x}^T(A(G_2)-A(G^*)){\bf z} \\&\geq x_{u_1}z_{v_{j_1}}+x_{v_{j_1}}z_{u_1}-\sum_{i=1}^{\zeta}\sum_{w\in N_{S''}(w_i)}(x_{w_i}z_w+x_wz_{w_i})-\sum_{w\in N_{S''}(u_1)}(x_{u_1}z_w+x_wz_{u_1})\\ &\geq x_{v_{j_1}}z_{u_1}-2\zeta\cdot \alpha_\mathbb{H}\cdot x_{S''}^*z_{S''}^*-\zeta \cdot x_{S''}^*z_{S''}^*\\ &\geq \frac{1}{2}x_u^*z_{u_1}-3\zeta\cdot  \alpha_\mathbb{H}\cdot\frac{x_u^*}{20C_\mathbb{H}}\cdot2z_{u_1}\\ &>0.
    \end{align*}
That is, $\rho(G_2)>\rho^*$, as desired. 
\end{proof}

Based on Claim~\ref{lem3.5} and the fact that $G^*\in {\rm SPEX}(n,\mathbb{H}_{\rm sub})$, we know that $G_2$ contains an $H$-subdivision for some $H\in\mathbb{H}$.
    Choose a minimal $H$-subdivision in $G_2$, say $H_S$, such that it contains as many vertices as possible in $S''$.
    Assume that $(P^1,\ldots,P^{m})$ is the corresponding $H$-partition of $H_S$, where $m=e(H)$. 
    
    Applying a similar discussion as Lemma \ref{lem3.3} yields that $S''\subseteq V(H_S)$. Recall that $P^0$ is a linear path inside $G_2$. Thus, $|V(P^i)\cap V(P^0)|\leq 2$ for every $i\in [m]$, which implies that at most $2m<2|H|^2\leq \zeta$ vertices in $P^0$ appear in $H_S$, a contradiction. 
\end{proof}

In the following, we complete the proof of Theorem~\ref{thm1.3}.
\begin{proof}[\bf Proof of Theorem \ref{thm1.3}]
Based on Lemma~\ref{lem2.4} and  Lemma~\ref{lem3.6}, we know that $V(G^*)=L\cup S$ with $|L|=\gamma_\mathbb{H}$ and $G^*[L,S]=K_{\gamma_\mathbb{H},n-\gamma_\mathbb{H}}$. Hence, in order to prove that $G^*$ contains $B_{\gamma_\mathbb{H},n-\gamma_\mathbb{H}}$ as a spanning subgraph, it suffices to show $G^*[L]\cong K_{\gamma_\mathbb{H}}$.

    Suppose to the contrary that $G^*[L]$ is not a clique. Without loss of generality, assume that $v_1v_2\notin E(G^*)$. 
    Since $|S|=n-|L|>\max_{H\in \mathbb{H}}\{2|H|^2\}=:\zeta$, one may choose a subset $T\subseteq S$ with $|T|=\zeta$.
    Let $T':=S\setminus T$.
    By Lemma~\ref{lem3.1}, $d_{S}(w)\leq\alpha_\mathbb{H}$ for each $w\in S$.
    It follows that 
    \begin{align*}
        e(S)-e(T')\leq \sum_{v\in T}d_{S''}(v)\leq \alpha_\mathbb{H}\zeta.
    \end{align*}
    Moreover, let $x_{S}^*:=\max_{w\in S}x_w$.
    Then $\rho^*x_{S}^*\leq\sum_{u\in L}x_u+\alpha_\mathbb{H}x_{S}^*$. Recall that $\rho^*\geq \sqrt{\gamma_\mathbb{H}(n-\gamma_\mathbb{H})}$.
    Hence,
    \begin{align*}
        x_{S}^*\leq\frac{\sum_{u\in L}x_u}{\rho^*-\alpha_\mathbb{H}}\leq\frac{\gamma_\mathbb{H}x_{u^*}}{\rho^*-\alpha_\mathbb{H}}\leq\sqrt{\frac{2\gamma_\mathbb{H}}{n}}x_{u^*}.
    \end{align*}

    Let $G$ be a graph obtained from $G^*+v_1v_2$ by removing all edges in $E(G^*[S])\setminus E(G^*[T'])$.
    By Lemma~\ref{lem2.4}, one has $\min\{x_{v_1},x_{v_2}\}\geq(1-\frac{1}{2(10C_\mathbb{H})^2})x_{u^*}$. Thus, as $n$ is sufficiently large, one has 
    \begin{align*}
        \rho(G)-\rho^*&\geq {\bf x}^T\left(A(G)-A(G^*)\right){\bf x}\\
        &\geq 2\left(x_{v_1}x_{v_2}-\sum_{w_1w_2\in E(G^*[S])\setminus E(G^*[T'])}x_{w_1}x_{w_2}\right)\\&\geq 2\left(\left(1-\frac{1}{2(10C_\mathbb{H})^2}\right)^2-\frac{\alpha_\mathbb{H}\zeta\cdot 2\gamma_\mathbb{H}}{n}\right)x_{u^*}^2\\&>0, 
    \end{align*}
    which implies that $\rho(G)>\rho^*$. Therefore, $G$ has an $H$-subdivision for some $H\in\mathbb{H}$, containing the edge $v_1v_2$. Choose a minimum  $H$-subdivision in $G$, say $H_S$. Assume that $(P^1,\ldots,P^{m})$ is the corresponding $H$-partition of $H_S$, where $m=e(H)$.

   Notice that $G[T]$ is an empty graph. By the minimality of $H_S$, we know that $|V(P^i)\cap T|\leq 2$ for each $i\in [m]$. Hence at least one vertex, say $w\in S$, does not appear in $H_S$. Consequently, $H_{S}$ is also an $H$-subdivision of $G-w$.
    Now, replace the edge $v_1v_2$ in $H_S$ with a three-vertex path $v_1wv_2$, which also forms an $H$-subdivision. It follows that $G^*$ contains an $H$-subdivision, a contradiction.
    
    This completes the proof of Theorem~\ref{thm1.3}.
\end{proof}

\section{Proof of Lemma \ref{lem2.4}}\label{sec4}

In this section, we give the proof of Lemma \ref{lem2.4}. Let $n$ be a sufficiently large integer and $\mathbb{H}$ be a family of graphs with $\gamma_{\mathbb{H}}\geq 1$.  Assume that $G$ is an $n$-vertex $\mathbb{H}$-subdivision-free graph with $\rho(G) \geq \sqrt{\gamma_\mathbb{H}(n-\gamma_\mathbb{H})}$.  Let ${\bf x}$ be a non-negative unit eigenvector of $A(G)$ corresponding to $\rho(G)$ and  $u^{*} \in V(G)$ be a vertex satisfying $x_{u^{*}}=\max_{u \in V(G)}x_{u}$.
To prove Lemma~\ref{lem2.4}, our goal is to find a set $L$ of exactly $\gamma_\mathbb{H}$ vertices in $G$ such that $x_{u} \geq (1-\frac{1}{2(10C_\mathbb{H})^{2}})x_{u^{*}}$ and $d_{G}(u) \geq (1-\frac{1}{(10C_\mathbb{H})^{2}})n$ for every $u \in L$.
For a positive integer $\lambda$,  define 
$$
L^\lambda=\{u\in V(G):x_{u}\geq (10C_\mathbb{H})^{-\lambda}x_{u^*}\}. 
$$
We adhere to the high-level methodological framework outlined in \cite{byrne2024general}:
\begin{enumerate}
    \item Show that $|L^4|$ is in fact bounded;
    \item Relate the eigenweight of vertices in $L^4$ to their degrees;
    \item Show that eigenweights of vertices in $L^1$  are close to $1$;
    \item Show that $|L^1|=k$.
\end{enumerate}

For a vertex $u\in V(G)$ and a positive integer $i$, let $N_i(u)$ be the set of vertices at distance $i$ from $u$ in $G$.
We will frequently use $N_1(u)$ and $N_2(u)$ below.
Furthermore, we use $L_i^\lambda(u)$ and $\overline{L}_i^\lambda(u)$ to denote $N_i(u)\cap L^\lambda$ and $N_i(u)\setminus L^\lambda$, respectively.
When the reference vertex $u$  is clear from context, we shall omit it and simply write $L_i^\lambda$ and $\overline{L}_i^\lambda$ respectively. 
We also denote $L_{i,j}^\lambda=L_i^\lambda \cup L_j^\lambda$ and $\overline{L}_{i,j}^\lambda=\overline{L}_i^\lambda \cup \overline{L}_j^\lambda$ for simplicity.

\begin{proof}[\bf Proof of Lemma \ref{lem2.4}]
With the above notation established, we are to find a set $L$ of exactly $\gamma_\mathbb{H}$ vertices in $G$ such that $x_{u} \geq (1-\frac{1}{2(10C_\mathbb{H})^{2}})x_{u^{*}}$ and $d_{G}(u) \geq (1-\frac{1}{(10C_\mathbb{H})^{2}})n$ for every $u \in L$. We first show an inequality for $\gamma_{\mathbb{H}}^2$.

\begin{claim}\label{lem4.2}
    For every vertex \( u \in V(G) \) and every positive integer \( \lambda \), we have 
    \begin{align}\label{eq:5}
        \gamma_\mathbb{H}(n-\gamma_\mathbb{H})x_u \leq |N_1(u)|x_u + \left( \frac{2C_\mathbb{H}n}{(10C_\mathbb{H})^{10-\lambda}} + \frac{2C_\mathbb{H}n}{(10C_\mathbb{H})^\lambda} \right)x_{u^*} + \sum_{\substack{v \in \overline{L}_1^\lambda,\,w \in N_{L_{1,2}^\lambda}(v)}} x_w.
    \end{align}
\end{claim}

\begin{proof}[\bf Proof of Claim \ref{lem4.2}]
    Recall that \( \rho(G) \geq \sqrt{\gamma_\mathbb{H}(n-\gamma_\mathbb{H})} \). Hence 
    \begin{align}\notag
        \gamma_\mathbb{H}(n-\gamma_\mathbb{H})x_u \leq \rho(G)^2x_u &= |N_1(u)|x_u + \sum_{\substack{v \in N_1(u),\,w \in N_1(v)\setminus\{u\}}} x_w\\\label{eq:6}
        &= |N_1(u)|x_u + \sum_{v\in L_1^\lambda \cup \overline{L}_1^\lambda}\,\sum_{\substack{w \in N_{L_{1,2}^\lambda}(v) \cup N_{\overline{L}_{1,2}^\lambda}(v)}} x_w.
    \end{align}   
   By the definition of \( L^\lambda \), vertices in \( \overline{L}_{1,2}^\lambda \) satisfy \( x_w < (10C_\mathbb{H})^{-\lambda}x_{u^*} \). Hence 
    \begin{align}\notag
        \sum_{v \in L_1^\lambda}\left(\sum_{w \in N_{L_{1,2}^\lambda}(v)} x_w + \sum_{w \in N_{\overline{L}_{1,2}^\lambda}(v)} x_w\right) 
        \leq &\sum_{v \in L_1^\lambda}\left(\sum_{w \in N_{L_{1,2}^\lambda}(v)} x_{u^*} + \sum_{w \in N_{\overline{L}_{1,2}^\lambda}(v)} (10C_\mathbb{H})^{-\lambda}x_{u^*}\right) \\\label{eq:7}
        \leq &\left(2e(L_1^\lambda) + e(L_1^\lambda, L_2^\lambda)\right)x_{u^*} + e\left(L_1^\lambda, \overline{L}_{1,2}^\lambda\right)(10C_\mathbb{H})^{-\lambda}x_{u^*},
    \end{align}
 and similarly, 
    \begin{align}\label{eq:8}
        \sum_{v \in \overline{L}_1^\lambda}\sum_{w \in N_{\overline{L}_{1,2}^\lambda}(v)} x_w\leq  \left(2e(\overline{L}_1^\lambda) + e(\overline{L}_1^\lambda, \overline{L_2^\lambda})\right)(10C_\mathbb{H})^{-\lambda}x_{u^*}.
    \end{align}
Together with Theorem~\ref{lem2.1} and the fact that $G$ is $\mathbb{H}$-subdivision-free, one has 
    \[
    2e(L_1^\lambda) + e(L_1^\lambda, L_2^\lambda) \leq 2C_\mathbb{H}|L^\lambda|, \quad \text{and} \quad e\left(L_1^\lambda, \overline{L}_{1,2}^\lambda\right) + 2e(L_1^\lambda) + e\left(\overline{L}_1^\lambda, \overline{L}_2^\lambda\right) < 2C_\mathbb{H}n.
    \]
    
Now, we consider the size of $|L^{\lambda}|$. Note that  $\rho(G) \geq \sqrt{\gamma_\mathbb{H}(n-\gamma_\mathbb{H})}>2C_\mathbb{H}(10C_\mathbb{H})^{10}$ as $n$ is sufficiently large and $x_u \geq (10C_\mathbb{H})^{-\lambda}x_{u^*}$ for every $u\in L^{\lambda}$. Then  
    \[
    |L^{\lambda}|2C_\mathbb{H}(10C_\mathbb{H})^{10-\lambda}x_{u^*} \leq \sum_{u\in L^{\lambda}}\rho(G)x_u = \sum_{u\in L^{\lambda}}\sum_{v \in N_1(u)} x_v \leq \sum_{u\in L^{\lambda}}|N_1(u)|x_{u^*}\leq 2e(G)x_{u^*} < 2C_\mathbb{H}nx_{u^*}.
    \]
    It follows that for every 
integer $\lambda$ we have
\begin{align}\label{eq:n01}
    |L^\lambda| < (10C_\mathbb{H})^{\lambda-10}n.
\end{align}
    By substituting \eqref{eq:7}-\eqref{eq:n01} into \eqref{eq:6}, we obtain our desired inequality immediately. 
\end{proof}
    
    Next, we show that $|L^4|$ is bounded. 
   
\begin{claim}\label{lem4.3}
    $|L^4| < (10C_\mathbb{H})^6$.
    \end{claim}

\begin{proof}[\bf Proof of Claim \ref{lem4.3}]
    Suppose, for contradiction, that there exists \( u_0 \in L^4 \) such that \( |N_1(u_0)| < (10C_\mathbb{H})^{-5}n \). Substituting \( \lambda = 5 \) in \eqref{eq:5} gives that
    \begin{align}\notag
        \gamma_\mathbb{H}(n-\gamma_\mathbb{H})x_{u_0} &\leq |N_1(u_0)|x_{u_0} +  \frac{4C_\mathbb{H}n}{(10C_\mathbb{H})^{5}} x_{u^*} + e(\overline{L}_1^5, L_{1,2}^5) x_{u^*}\\\notag
        &\leq \left((10C_\mathbb{H})^{-5}n+\frac{4C_\mathbb{H}n}{(10C_\mathbb{H})^{5}}+C_\mathbb{H}(|N_1(u_0)| + |L^5|)\right)x_{u^*}\\\label{eq111}
        &\leq \frac{(1+6C_\mathbb{H})n}{(10C_\mathbb{H})^5}x_{u^*},
    \end{align}
 the last inequality holds by \eqref{eq:n01}. 

    On the other hand, since \( \gamma_\mathbb{H} \geq 1 \) and \( x_{u_0} \geq (10C_\mathbb{H})^{-4}x_{u^*} \), we have
    \[
    \gamma_\mathbb{H}(n-\gamma_\mathbb{H})x_{u_0} \geq \frac{7}{10}nx_{u_0} \geq 7C_\mathbb{H}(10C_\mathbb{H})^{-5}nx_{u^*},
    \]
    which contradicts~\eqref{eq111}. Thus \( |N_1(u)| \geq (10C_\mathbb{H})^{-5}n \) for all \( u \in L^4 \).
    Summing this over \( L^4 \) and applying Theorem~\ref{lem2.1} yields 
    \[
    |L^4|(10C_\mathbb{H})^{-5}n \leq \sum_{u \in L^4} |N_1(u)| < 2C_\mathbb{H}n.
    \]
    Thus,  $|L^4| < (10C_\mathbb{H})^6$, as desired.
\end{proof}
  We now estimate the degrees of vertices in $L^4$
  with respect to their eigenweights.
\begin{claim}\label{lem4.4}
    For each \( u \in L^4 \), we have 
    $|N_1(u)| \geq \left(\frac{x_u}{x_{u^*}} - \frac{1}{(10C_\mathbb{H})^3}\right)n$.
    \end{claim}

\begin{proof}[\bf Proof of Claim \ref{lem4.4}]
    Given a vertex $u\in L^4$, let $L_0=\{v\in \overline{L}_1^4: |N(v)\cap L_{1,2}^4|\geq \gamma_{\mathbb{H}}\}$ and $H^*$ be a minimal  graph with respect to $\mathbb{H}$. 
    We first claim that 
    $|L_0|\leq \eta(\alpha_\mathbb{H}+\binom{|H^*|+1}{2})$, where $\eta:=\binom
    {|L_{1,2}^4|}{\gamma_\mathbb{H}}$.
    In fact, if $|L_{1,2}^4|\leq \gamma_\mathbb{H}-1$, then $L_0=\emptyset$ and we are done.
    Now we  consider the case $|L_{1,2}^4|\geq\gamma_H$.
    Suppose in contrast that $|L_0|>\eta(\alpha_\mathbb{\mathbb{H}}+\binom{|H^*|+1}{2})$.
    Since there are only $\eta$ options for vertices in $L_0$ to choose $\gamma_\mathbb{H}$ neighbors from $L_{1,2}^4$, we can find $\gamma_\mathbb{H}$ vertices in $L_{1,2}^4$ with at least $|L_0|/\eta>\alpha_\mathbb{H}+\binom{|H^*|+1}{2}$ common neighbors in $L_0$.
    Furthermore, $u\notin L_{1,2}^4$ and $L_0\subseteq \overline{L}_1^4\subseteq N_1(u)$.
    Hence $u$ and those $\gamma_\mathbb{H}$ vertices have $\alpha_\mathbb{H}+\binom{|H^*|+1}{2}$ common neighbors, which implies that $G$ contains a bipartite subgraph $G[S,T]$ isomorphic to $K_{\gamma_\mathbb{H}+1,\alpha_\mathbb{H}+\binom{|H^*|+1}{2}}$. Thus, by Lemma~\ref{lem2.2}, $G$ contains an $H^*$-subdivision, a contradiction. 
    Therefore, $|L_0|\leq \eta(\alpha_\mathbb{H}+\binom{|H^*|+1}{2})$.
    
    By Claim~\ref{lem4.3} and the definition of $L_0$, we have 
    \begin{align}\label{eq:10}
        e(\overline{L}_1^4, L_{1,2}^4) =e(L_0, L_{1,2}^4)+e(\overline{L}_1^4\setminus L_0, L_{1,2}^4) \leq (10C_\mathbb{H})^{-4}n+|N_1(u)|(\gamma_\mathbb{H}-1).
    \end{align}
    Substituting \(\lambda = 4\) in 
    \eqref{eq:5} yields \allowdisplaybreaks
    \begin{align*}
    \gamma_\mathbb{H}(n-\gamma_\mathbb{H})x_u &\leq |N_1(u)|x_u + \left( \frac{2C_\mathbb{H}n}{(10C_\mathbb{H})^{6}} + \frac{2C_\mathbb{H}n}{(10C_\mathbb{H})^4} \right)x_{u^*} + \sum_{\substack{v \in \overline{L}_1^4,\,w \in N_{L_{1,2}^4}(v)}} x_w\\
    &\leq \left(|N_1(u)| +  \frac{3C_\mathbb{H}n}{(10C_\mathbb{H})^{4}}+e(\overline{L}_1^4, L_{1,2}^4)\right)x_{u^*}\\
    &\leq \left(|N_1(u)| +  \frac{4C_\mathbb{H}n}{(10C_\mathbb{H})^{4}}+|N_1(u)|(\gamma_\mathbb{H}-1)\right)x_{u^*}\\
    &\leq \gamma_{\mathbb{H}}\left(|N_1(u)| +  \frac{4C_\mathbb{H}n}{(10C_\mathbb{H})^{4}}\right)x_{u^*}.
    \end{align*}
    This implies 
    \[
    |N_1(u)| \geq (n-\gamma_{\mathbb{H}})\frac{x_u}{x_{u^*}}-\frac{4C_\mathbb{H}n}{(10C_\mathbb{H})^{4}}=\left(\frac{x_u}{x_{u^*}}-\frac{4C_\mathbb{H}}{(10C_\mathbb{H})^{4}}\right)n-\gamma_{\mathbb{H}}\frac{x_u}{x_{u^*}}\geq \left(\frac{x_u}{x_{u^*}}-\frac{1}{(10C_\mathbb{H})^{3}}\right)n, 
    \]
    the last inequality holds as $\gamma_{\mathbb{H}}\frac{x_u}{x_{u^*}}\leq \gamma_{\mathbb{H}}\leq \frac{6C_\mathbb{H}n}{(10C_\mathbb{H})^{4}}$. 
\end{proof}

     Finally we consider the eigenweights and degrees of vertices in $L^1$. 
    
\begin{claim}\label{lem4.5}
    For every \( u \in L^1 \), we have
    $x_u \geq \left(1-\frac{1}{2(10C_\mathbb{H})^2}\right)x_{u^*}$ and $|N_1(u)| \geq \left(1-\frac{1}{(10C_\mathbb{H})^2}\right)n$.
    \end{claim}

\begin{proof}[\bf Proof of Claim \ref{lem4.5}]
    Suppose to the contrary that  there is a vertex \( u_0 \in L^1 \) satisfying \( x_{u_0} < \left(1-\frac{1}{2(10C_\mathbb{H})^2}\right)x_{u^*} \). It is obvious that $L^1\subseteq L^4$ and $u^*\in L^1$.  By Claim~\ref{lem4.4} and the definition of \( L^1 \), we have
    \[
    |N_1(u^*)| \geq \left(1-\frac{1}{(10C_\mathbb{H})^3}\right)n\ \ \text{and}\ \ 
    |N_1(u_0)| \geq \left(\frac{1}{10C_\mathbb{H}} - \frac{1}{(10C_\mathbb{H})^3}\right)n.
    \]
    Together with Claim~\ref{lem4.3} one has $|\overline{L}_1^4(u^*)| \geq \left(1 - \frac{2}{(10C_\mathbb{H})^3}\right)n$. Thus, 
    \begin{align}\label{eq:L}
        |\overline{L}_1^4(u^*) \cap N_1(u_0)| \geq \frac{9n}{100C_\mathbb{H}}.
    \end{align}
It follows that \( u_0 \in L_{1,2}^4(u^*) \). Substituting \( \lambda = 4 \) in \eqref{eq:5} implies\allowdisplaybreaks
\begin{align*}
    \gamma_\mathbb{H}(n-\gamma_\mathbb{H})x_{u^*} &\leq \left(|N_1({u^*})| +  \frac{3C_\mathbb{H}n}{(10C_\mathbb{H})^{4}} +e\left(\overline{L}_1^4(u^*), L_{1,2}^4(u^*)\setminus \{u_0\}\right) \right)x_{u^*}+e\left(\overline{L}_1^4(u^*), \{u_0\}\right)x_{u_0} \\
    &\leq \left(|N_1({u^*})| +  \frac{3C_\mathbb{H}n}{(10C_\mathbb{H})^{4}} +e\left(\overline{L}_1^4(u^*), L_{1,2}^4(u^*)\right) \right)x_{u^*}+e\left(\overline{L}_1^4(u^*), \{u_0\}\right)(x_{u_0}-x_{u^*})\\
    &\leq 
    \left(\gamma_\mathbb{H}|N_1({u^*})| +  \frac{4C_\mathbb{H}n}{(10C_\mathbb{H})^{4}}\right)x_{u^*}-\frac{9n}{100C_\mathbb{H}}\cdot \frac{x_{u^*}}{2(10C_{\mathbb{H}})^2},
\end{align*}
the last inequality follows by  \eqref{eq:10}, \eqref{eq:L} and our assumption for $x_{u_0}$. Furthermore, notice that $\gamma_{\mathbb{H}}^2\leq \frac{C_{\mathbb{H}}n}{3(10C_{\mathbb{H}})^4}$ as $n$ is sufficiently large. Therefore,
$$
\gamma_\mathbb{H}n \leq \left(\gamma_\mathbb{H}|N_1({u^*})| +  \frac{4C_\mathbb{H}n}{(10C_\mathbb{H})^{4}}\right)-\frac{9n}{100C_\mathbb{H}}\cdot \frac{1}{2(10C_{\mathbb{H}})^2}+\frac{C_{\mathbb{H}}n}{3(10C_{\mathbb{H}})^4}<\gamma_{\mathbb{H}} n,
$$
a contradiction. 
    Therefore,  \( x_u \geq \left(1-\frac{1}{2(10C_\mathbb{H})^2}\right)x_{u^*} \) for every $u\in L^1$. 
    
    Recall that $L^1\subseteq L^4$. It follows from Claim~\ref{lem4.4} that 
    \[
    |N_1(u)| \geq \left(1 - \frac{1}{2(10C_\mathbb{H})^2} - \frac{1}{(10C_\mathbb{H})^3}\right)n \geq \left(1 - \frac{1}{(10C_\mathbb{H})^2}\right)n.\qedhere
    \]
\end{proof}

Now we know that the vertices in $L^1$ possess the properties that we desire, the remaining work is to prove that there are exactly $\gamma_\mathbb{H}$ vertices in $L^1$.

\begin{claim}\label{lem4.6}
     \( |L^1| = \gamma_\mathbb{H} \).
\end{claim}

\begin{proof}
    We first suppose that \( |L^1| \geq \gamma_\mathbb{H} + 1 \) and choose $L'\subseteq L^1$ with $|L'|=\gamma_\mathbb{H} + 1=\gamma_{H^*} + 1$, where $H^*$ is a minimal graph with respect to $\mathbb{H}$. By applying Claim \ref{lem4.5}, we obtain  
    \[
    \left|\bigcap_{u \in L'} N(u)\right|\geq n - \frac{(\gamma_{H^*}+1)n}{(10C_\mathbb{H})^2} \geq \alpha_{H^*} + \binom{\gamma_{H^*}+1}{2}.
    \]
    This forces a complete bipartite subgraph \( K_{\gamma_{H^*}+1, \alpha_H + \binom{\gamma_{H^*}+1}{2}} \). Combining  Lemma \ref{lem2.2}, $G$ contains an $H^*$-sudivision, a contradiction. Hence  \( |L^1| \leq \gamma_\mathbb{H}\). 
    
   Next, suppose that $|L^1|\leq \gamma_{\mathbb{H}}-1$. In the following, we consider the neighbors of $u^*$ in $L^4$.  Since  \( u^* \in L^1 \setminus L_{1,2}^4 \), one has $|L^1\cap L_{1,2}^4|\leq \gamma_{\mathbb{H}}-2$. Therefore, $e(\overline{L}_1^4, L^1 \cap L_{1,2}^4) \leq (\gamma_\mathbb{H}-2)|N_1(u^*)|$. 
    The definition of $L^1$ implies that  $x_w\leq \frac{x_{u^*}}{10C_{\mathbb{H}}}$  for every vertex $w\in  L_{1,2}^4\setminus L^1$.  Substituting \( \lambda = 4 \) in \eqref{eq:5} gives that
    \begin{align*}
    \gamma_\mathbb{H}(n-\gamma_\mathbb{H})x_{u^*} &\leq \left(|N_1({u^*})| +  \frac{3C_\mathbb{H}n}{(10C_\mathbb{H})^{4}} +e\left(\overline{L}_1^4,L^1 \cap L_{1,2}^4\right) \right)x_{u^*}+e\left(\overline{L}_1^4,L_{1,2}^4\setminus L^1\right)\frac{x_{u^*}}{10C_{\mathbb{H}}} \\
    &\leq \left((\gamma_{\mathbb{H}}-1)|N_1({u^*})| +  \frac{3C_\mathbb{H}n}{(10C_\mathbb{H})^{4}}\right)x_{u^*}+C_{\mathbb{H}}n\cdot \frac{x_{u^*}}{10C_{\mathbb{H}}}\\
    &\leq \left((\gamma_{\mathbb{H}}-1) +  \frac{3C_\mathbb{H}}{(10C_\mathbb{H})^{4}}+\frac{1}{10}\right)nx_{u^*}\\
    &\leq \left(\gamma_{\mathbb{H}}-\frac{4}{5}\right)nx_{u^*},
\end{align*}
the second inequality follows by Theorem  \ref{lem2.1}. This yields $\gamma_{\mathbb{H}}^2\geq \frac{4}{5}n$, which is a contradiction as $n$ is sufficiently large. Thus, \( |L^1| = \gamma_\mathbb{H} \), as desired.
\end{proof}

Combining Claim~\ref{lem4.5} and Claim~\ref{lem4.6}, we immediately obtain the desired result by choosing $L=L^1$.
\end{proof}
\section{More characterization of ${\rm SPEX}(n,\mathbb{H}_{\rm sub})$}\label{sec5}

Building on Theorem~\ref{thm1.3}, we shall focus on the characterization of $G^*-L$.
Before this, we first introduce some new notations.
Let $\mathbb{H}$ be a family of graphs and $s$ be an integer. Choosing an arbitrary graph $H \in \mathbb{H}$, we define a family of induced subgraphs of $H$ as follows: 
\begin{align*}
    \Gamma_s(H)=\{H[S]:S\subseteq V(H), |S|=s\}.
\end{align*}
A member $H[S]$ in $\Gamma_s(H)$ is said to be \textit{irreducible}, if $\Gamma_s(H)$ does not contain any member isomorphic to a proper subgraph of $H[S]$.
Now let $\Gamma^*_s(H)$ be the family of $s$-vertex irreducible induced subgraphs of $H$ and
\begin{align*} \Gamma(\mathbb{H})=\bigcup_{H\in\mathbb{H}}\Gamma^*_{|H|-\gamma_\mathbb{H}}(H).
\end{align*}

With the notations above, we have the following result.
\begin{theorem}
\label{cor5.1}
Let $\mathbb{H}$ be a family of graphs and $n$ be a sufficiently large integer.  If $G\in {\rm SPEX}(n,\mathbb{H})$, then there is a vertex subset $L$ of size $\gamma_{\mathbb{H}}$ such that  $G-L$ 
is $\Gamma(\mathbb{H})$-subdivision-saturated. 
\end{theorem}

\begin{proof}
In view of Theorem \ref{thm1.3}, we know that $G$ contains a spanning subgraph isomorphic to $B_{\gamma_{\mathbb{H}},n-\gamma_{\mathbb{H}}}$. Let $L\subseteq V(G)$ be the set of the $\gamma_\mathbb{H}$  vertices with degree $n-1$. We claim that $G$ is $\mathbb{H}$-subdivision-free if and only if $G-L$ is $\Gamma(\mathbb{H})$-subdivision-free.
    
    Suppose  that $G$ is $\mathbb{H}$-subdivision-free but  $G-L$ contains an $H_0$-subdivision for some $H_0\in \Gamma(\mathbb{H})$.
    By the definition of $\Gamma(\mathbb{H})$, there exists an $H\in \mathbb{H}$ such that $H_0$ is an $(|H|-\gamma_\mathbb{H})$-vertex induced subgraph of $H$.
    Note that each vertex in 
$L$ is adjacent to all other vertices in $G$. Hence we can find an $H$-subdivision in $G$, a contradiction. 

    Conversely, suppose  that $G-L$ is $\Gamma(\mathbb{H})$-subdivision-free but $G$ contains an $H$-subdivision for some $H\in\mathbb{H}$.
    Let $H_S$ be a minimal $H$-subdivision of $G$ such that it contains as many vertices as possible in $L$. Clearly, $L\subseteq V(H_S)$. 
    As $H_S$ is minimal, every vertex in $L$ can not be a subdivision vertex of $H_S$. 
    Thus,  $H_S-L$ contains an $H_0$-subdivision for some $H_0\in \Gamma_{|H|-\gamma_\mathbb{H}}^*(H)$. 
    Thus, $G-L$ contains an $H_0$-subdivision, a contradiction.  


 Notice that $G$ is connected. 
    Hence, adding an arbitrary  edge between two non-adjacent vertex pairs in $G$ will increase its spectral radius.
    Furthermore, by the claim above, we know that $G^*-L$ is $\Gamma(\mathbb{H})$-subdivision-free.
    Thus, $G^*-L$ is $\Gamma(\mathbb{H})$-subdivision-saturated, as desired. 
\end{proof}

\section{Concluding remarks}
In this paper, we characterize the structure of $n$-vertex  $\mathbb{H}$-subdivision-free graphs having the maximum spectral radius. We show that for sufficiently large $n$, every graph $G\in {\rm SPEX}(n,\mathbb{H}_{\rm sub})$ contains a spanning subgraph isomorphic to the generalized book $B_{\gamma_\mathbb{H},n-\gamma_\mathbb{H}}$. Denote by $L$ the set consisting of all  dominating vertices in $B_{\gamma_\mathbb{H},n-\gamma_\mathbb{H}}$. We also prove that $G-L$ is $\Gamma(\mathbb{H})$-subdivision-saturated. This develops a spectral framework for subdivision problems analogous to Zhai, Fang, and Lin's result  \cite{zhai} in the minor-free setting. 

There are still some other interesting problems to be considered along the above line. For example, a precise structural characterization of $G-L$ remains open. 
In particular, determining the exact structure of graphs in $\mathrm{SPEX}(n,\{K_{s,t}\}_{\mathrm{sub}})$ would provide a valuable counterpart to the minor-free case (see \cite{Tait2019,Tait2017,Zhai-minor}).


\bibliographystyle{abbrv}
\bibliography{ref}
\end{document}